\documentclass[12pt,reqno]{amsart}
\usepackage{fullpage}
\usepackage{amsfonts}
\usepackage{amssymb}
\usepackage{times}
\usepackage{graphicx}
\usepackage{blkarray}
\usepackage{algorithm}
\usepackage[noend]{algpseudocode}
\usepackage{xcolor}
\usepackage{stmaryrd}

\newtheorem{theorem}{Theorem}
\newtheorem{corollary}{Corollary}
\newtheorem{lemma}{Lemma}

{ \theoremstyle{remark} }

\newcommand{\Prob}{{\mathbb{P}}}
\usepackage{tikz}
\usetikzlibrary{shapes.geometric, arrows}
\tikzstyle{RV} = [circle, minimum width=1cm,text centered, draw=black]
\tikzstyle{dist} = [rectangle,rounded corners, text centered, draw=black]
\tikzstyle{arrow} = [thick,->,>=stealth]
\tikzstyle{desc} = [thick,dashed]
\tikzstyle{line} = [thick]

\begin{document}

\title{Hidden independence in unstructured probabilistic models}
\author{Antony Pearson \and Manuel E. Lladser}
\email{manuel.lladser@colorado.edu}
\maketitle

\begin{abstract}
We describe a novel way to represent the probability distribution of a random binary string as a mixture having a maximally weighted component associated with independent (though not necessarily identically distributed) Bernoulli characters. We refer to this as the \textit{latent independent weight} of the probabilistic source producing the string, and derive a combinatorial algorithm to compute it. The decomposition we propose may serve as an alternative to the Boolean paradigm of hypothesis testing, or to assess the fraction of uncorrupted samples originating from a source with independent marginals. In this sense, the latent independent weight quantifies the maximal amount of independence contained within a probabilistic source, which, properly speaking, may not have independent marginals.
\end{abstract}

\section{Introduction}
\label{sec:intro}

Consider the Bayesian network~\cite{probGraphMod} in Figure~\ref{fig:BN96indep}, given in~\cite[Chapter 2]{graphmod}. As the reader may find familiar, each random variable (node) in the network, given the configurations of its parents, is by definition conditionally independent from its non-descendants. Accordingly, the joint probability mass function of the binary random vector $(P,T,S,L,X)$ factorizes as follows:
\[\Prob(P,T,S,L,X) = \Prob(P)\cdot \Prob(T)\cdot \Prob(S\mid T)\cdot \Prob(L\mid P,T)\cdot \Prob(X\mid L).
\]
In particular, the joint distribution of $P$, $T$, $S$, $L$ and $X$ can be encoded with $10$ free parameters. Perhaps unexpectedly, however, one can represent this joint distribution as a mixture with a heavily weighted ``independent'' component. Specifically:
\begin{equation}
\Prob = 0.94\cdot Be(0.02)\otimes Be(0.005)\otimes Be(0.6)\otimes Be(0.01)\otimes Be(0.6) + 0.06\cdot R,
\label{ide:TubPne}
\end{equation}
where $Be(p)$ denotes a Bernoulli distribution with success probability $p$, the operator $\otimes$ denotes a product of Bernoulli distributions (i.e. the joint distribution of independent Bernoulli random variables), and $R$ is a ``residual'' probability distribution over the sample space $\{0,1\}^5$. This decomposition of $\Prob$ is possible because for each outcome $(p,t,s,l,x)\in\{0,1\}^5$ a  computation shows that: 
\[\Prob(p,t,s,l,x)\ge0.94\,\cdot\,0.02^p\,0.98^{1-p}\times0.005^t\,0.995^{1-t}\times0.6^l\,0.4^{1-l}\times0.01^x\,0.99^{1-x}\times0.6^s\,0.4^{1-s}.\]
In particular, $R$ may be obtained solving for it in equation~(\ref{ide:TubPne}). It turns out in this case that $R$ has low entropy ($\approx3.2$ bits, compared to the uniform distribution over $\{0,1\}^5$, which has $5$ bits of entropy), and gives probability $0$ to twelve of the thirty-two outcomes.

The identity in equation~(\ref{ide:TubPne}) means that, conditioned on a hidden event of $94\%$ probability, the presence of lung infiltrates, the outcome of an X-ray and sputum smear, and the status of a patient having tuberculosis or pneumonia will all be rendered independent. Thus, while in a clinical setting the dependencies encoded in the Bayesian network may be relevant, on the population level, these covariates often behave independently. That is, despite the intricate dependencies encoded in the Bayesian network, most samples from this model can be attributed to a much simpler model (with $5$ instead of $10$ free parameters).

The decomposition in (\ref{ide:TubPne}) bears the question: \textit{what's the largest  weight a product of independent Bernoulli's can have as component of $\Prob$?} Remarkably, the marginal distributions of $\Prob$ are associated with a weight that is significantly smaller than 94\%. Indeed, a computation shows that $P\sim Be(0.05)$, $T\sim Be(0.02)$, $S\sim Be(0.6)$, $L\sim Be(0.05)$, and $X\sim Be(0.6)$, and that $P$ admits the mixture representation:
\[\Prob = \epsilon\cdot Be(0.05)\otimes  Be(0.02) \otimes Be(0.6)\otimes  Be(0.05) \otimes Be(0.6) + (1-\epsilon)\cdot R,\]
where $\epsilon\approx0.104$, and $R$ is a probability distribution that can be determined from the above identity.

\begin{figure}
\centering
\begin{tikzpicture}[node distance=1cm]
\node (L)[RV]{$L$};
\node (P) [RV, above of=L,xshift=-2cm,yshift=1cm]{$P$};
\node (T) [RV,above of=L,xshift=2cm,yshift=1cm]{$T$};
\node (X)[RV,below of=L,xshift=-2cm,yshift=-1cm]{$X$};
\node (S)[RV,below of=L,xshift=2cm,yshift=-1cm]{$S$};
\node (Pd)[dist,above of=P,yshift=0.5cm]{$\mathbb{P}(P=1)=0.05$};
\node (Td)[dist,above of=T,yshift=0.5cm]{$\Prob(T=1)=0.02$};
\node(Sd)[dist,below of=S,xshift=1cm,yshift=-1cm]{
\begin{tabular}{ c | c }
 $t$ & $\Prob(S=1\mid T=t)$\\
 \hline
 $0$ & $0.6$ \\  
 $1$ & $0.8$  
 \end{tabular}};
\node(Xd)[dist,below of=X,xshift=-1cm,yshift=-1cm]{
\begin{tabular}{ c | c }
 $l$ & $\Prob(X=1\mid L=l)$\\
 \hline
 $0$ & $0.6$ \\  
 $1$ & $0.8$  
 \end{tabular}};
\node(Ld)[dist,left of=L,xshift=-3.5cm]{
\begin{tabular}{ c |c| c }
 $p$ & $t$ & $\Prob(L=1\mid P=p,T=t)$\\
 \hline
 $0$ & $0$ & $0.01$ \\  
 $0$ & $1$ & $0.2$\\
 $1$ & $0$ & $0.6$\\
 $1$ & $1$ & $0.8$
 \end{tabular}};
\draw [arrow] (P)--(L);
\draw [arrow] (T)--(L);
\draw [arrow] (L)--(X);
\draw [arrow] (T)--(S);
\draw [desc] (P)--(Pd);
\draw [desc] (T)--(Td);
\draw [desc] (L)--(Ld);
\draw [desc] (X)--(Xd);
\draw [desc] (S)--(Sd);
\end{tikzpicture}
\caption{Bayesian network that models the interaction between two lung conditions, tuberculosis (T), and pneumonia (P), and how they jointly affect the probability that a patient will have lung infiltrates (L), the presence of said infiltrates in an X-ray (X), and the outcome of a sputum smear test (S) for tuberculosis. Nodes represent Bernoulli random variables, with conditional probability tables indicated, and the value $1$ ($0$) indicates the presence (absence) of the corresponding condition.}
\label{fig:BN96indep}
\end{figure}
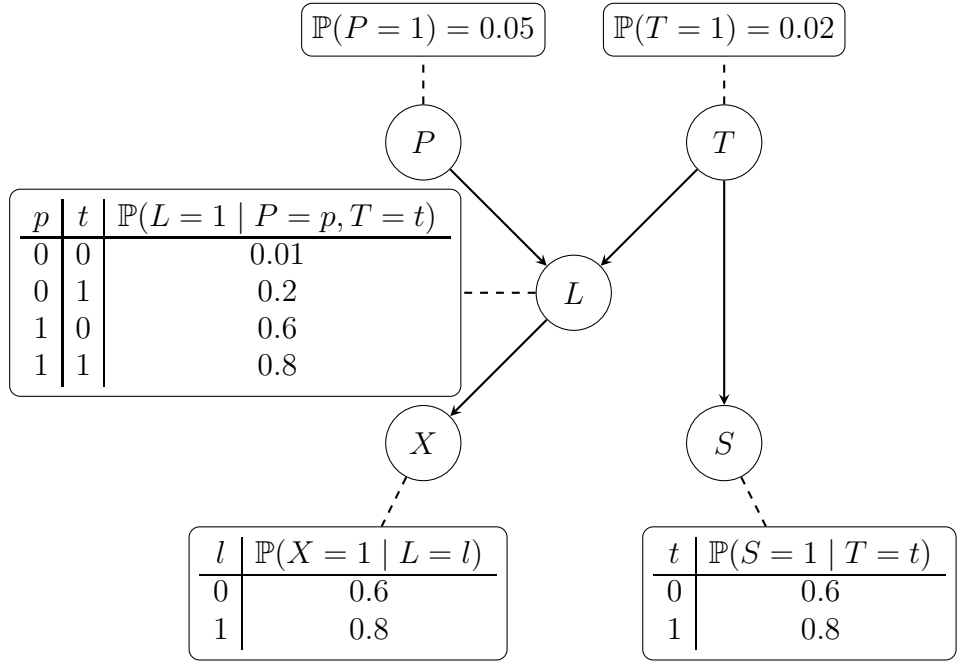

In this article we develop the mathematics of the so-called \emph{independent weight} of an arbitrary joint probability distribution over a sample space of the form $\{0,1\}^d$, with $d\ge1$ finite. We argue that the independent weight of a probabilistic source describes the largest average fraction of samples from the source that can be attributed to (conditionally) independent Bernoulli random variables, and describe an algorithm to compute this weight, along with some heuristics to approximate it. 

The independent weight of a probabilistic source is, therefore, an intrinsic property of it, which can be used as an objective measure of the truthiness of the null hypothesis that ``the source has independent marginals,'' which may be nevertheless false (as the example associated with Figure~\ref{fig:BN96indep}). The concept of independent weight may also be used to distill corrupted data from a source with otherwise independent marginals.

\subsection{Related Work}

The present work may be regarded as a non-trivial specialization of the recent theory developed in~\cite{PeaLla20}. This previous work introduces the concept of the \emph{latent weight} of a probabilistic source (such as $\Prob$ in the previous example) with respect to a structured class $\mathcal{Q}$ of probability models over a finite sample space. Specifically, the latent weight of a source $P$ with respect to a class $\mathcal Q$ of models is defined as~\cite{PeaLla20}:
\begin{align}
    \lambda_{\mathcal{Q}}(P):=\sup\{\lambda\geq0\mid P\geq\lambda\cdot Q\text{ for some }Q\in\mathcal Q\}.
\end{align}
This coefficient represents the largest weight that can be given to a model in  $\mathcal Q$ as a component in a mixture decomposition of $P$. In fact, under mild technical conditions, there always exists $Q\in\mathcal Q$ and a probabilistic model $R$ such that
\begin{align}
P=\lambda_{\mathcal Q}(P)\cdot Q+(1-\lambda_{\mathcal Q}(P))\cdot R.
\end{align}
Furthermore, when $\mathcal{Q}$ is convex, $Q$ is unique when $\lambda_{\mathcal Q}(P)>0$, and so is $R$ when $\lambda_{\mathcal Q}(P)<1$.


In the current setting, $\mathcal Q$ is the class of probability distributions associated with independent binary random variables. We emphasize that much of what we present in this extended abstract may be generalized to more general discrete random variables, however, the binary setting presents enough mathematical challenges to consider it in isolation.

\section{Latent Independent Weights}

In what follows, $\mathcal{P}$ denotes the set of all probability distributions on $\{0,1\}^d$, with $d\ge1$ a given integer. In particular, we may think of elements in $\mathcal{P}$ as a non-negative real vectors of dimension $2^d$, with entries that sum up to 1.

For $P,Q\in\mathcal{P}$ and $\lambda\in\mathbb{R}$, we write $P\ge\lambda\cdot Q$ to mean that $P(\omega)\geq \lambda\cdot Q(\omega)$, for each $\omega\in\{0,1\}^d$. Further, we say that $Q$ has \emph{independent marginals} if and only if there are probability distributions $\mu_1,\ldots,\mu_d$ defined over $\{0,1\}$ such that $Q=\otimes_{i=1}^d\mu_i$. Equivalently, $Q$ has independent marginals if and only if it is the probability distribution of a random vector of the form $(X_1,\ldots,X_d)$, with $X_1,\ldots,X_d$ independent (though not necessarily identically distributed) Bernoulli random variables. (In this case, each $X_i$ has distribution $\mu_i$.)

We associate to each $P\in\mathcal{P}$ the real coefficient: 
\begin{equation}
\label{eq:indWeight}
\lambda(P):= \sup\{\lambda\geq0\mid\exists\,Q\text{ with independent marginals such that }P\geq \lambda\cdot Q\}.
\end{equation}

Clearly, $0\le\lambda(P)\le1$. In fact, according  to~\cite{PeaLla20}, $\lambda(P)=1$ if and only if $P$ has independent marginals itself. Furthermore, because the subset of distributions in $\mathcal{P}$ with independent marginals is compact, the supremum in equation (\ref{eq:indWeight}) is always achieved~\cite{PeaLla20}. Namely, there is $Q\in\mathcal{P}$ with independent marginals such that $P\ge\lambda(P)\cdot Q$. As a result, since $(P-\lambda(P)\cdot Q)$ is a measure with total mass $(1-\lambda(P))$, there is also $R\in\mathcal P$ such that $P$ admits the mixture decomposition:
\begin{equation}
\label{eq:mix}
P=\lambda(P)\cdot Q+\big(1-\lambda(P)\big)\cdot R.
\end{equation}
This decomposition motivates calling $\lambda(P)$ the \emph{latent independent weight of} $P$, or simply the \emph{independent weight of} $P$. It follows that \emph{$\lambda(P)$ is the largest weight that can be attributed to a probability measure over $\{0,1\}^d$ with independent marginals as a component of $P$}. Equivalently: \emph{$\lambda(P)$ is  the maximal expected fraction of samples from $P$ which may be attributed to a probabilistic source with independent marginals.} More precisely, if $X=(X_1,\ldots,X_d)$ has distribution $P$ then, up to a hidden event with probability $\lambda(P)$, the Bernoulli random variables $X_1,\ldots,X_d$ are (conditionally) independent.

We note that the model $Q$ with independent marginals in equation~(\ref{eq:mix}) is not necessarily unique. For example, let $d=2$ and $P$ be the uniform distribution over $\{(0,0),(1,1)\}$; in particular,  $P=\delta_{(0,0)}/2+\delta_{(1,1)}/2$, where $\delta_{x}$ is the point probability mass at $x$. Careful analysis can verify that $\lambda(P)=1/2$, hence the supremum in equation~(\ref{eq:indWeight}) is achieved by $\delta_{(0,0)}$ as well as $\delta_{(1,1)}$, both of which have independent marginals.

\subsection{Alternative Formulations}

In this section we show how to compute latent independents weights.

Henceforth, $P\in\mathcal{P}$ is assumed fixed. Moreover, we assume that $P>0$, i.e. $P(\nu)>0$ for each $\nu\in\{0,1\}^d$. This assumption can be relaxed but goes beyond the scope of this extended abstract.

For $\omega=(\omega_1,\ldots,\omega_d)\in\{0,1\}^d$ and $q=(q_1,\ldots,q_d)\in[0,1]^d$ define
\begin{align*}
f_\omega (q) :=\prod_{i=1}^d q_i^{-\omega_i}(1-q_i)^{\omega_i-1}.
\end{align*}

\begin{lemma}
\label{lem:cont}
For each $\omega\in\{0,1\}^d$, the function $f_\omega:[0,1]^d\to\mathbb{R}\cup\{+\infty\}$ is continuous.
\end{lemma}
\begin{proof}
Fix an $\omega$. Observe that if $X=(X_1,\ldots,X_d)$ is a random vector with independent entries such that $X_i\sim Bernoulli(q_i)$, and $q=(q_1,\ldots,q_d)$, then \[\Prob(X=\omega\,|\,q)=\prod_{i=1}^d\Prob(X_i=\omega_i\,|\,q)=\frac{1}{f_\omega(q)}\]
i.e.
\[f_\omega(q)=\frac{1}{\Prob(X=\omega\,|\,q)}.\]
In particular, since $\Prob(X=\omega\mid q)$ is a continuous function of the parameter $q$, the lemma follows.
\end{proof}

In what follows, for each $w\in\{0,1\}^d$, we define
\begin{equation}
\label{def:Qomega}
\mathcal{Q}_\omega:=\left\{q\in[0,1]^d\,|\,\forall\nu\in\{0,1\}^d:\,P(\omega) f_\omega(q)\leq P(\nu) f_\nu (q)\right\}.
\end{equation}


\begin{lemma}
\label{lem:formula}
If $P>0$ then $\lambda(P)=\max\limits_{\omega \in\{0,1\}^d}\,\max\limits_{q\in \mathcal{Q}_\omega }P(\omega) f_\omega (q)$.
\end{lemma}
\begin{proof}
Since a probability measure over $\{0,1\}^d$ with independent marginals may be represented in terms of $d$ independent Bernoulli random variables, we may restate equation~(\ref{eq:indWeight}) equivalently as follows:
\begin{align*}
\lambda(P) =& \sup\left\{\lambda\ge0\mid\exists q\in[0,1]^d\,\forall\nu\in\{0,1\}^d: P(\nu)\geq \lambda\cdot\prod_{i=1}^d q_i^{\nu_i}(1-q_i)^{1-\nu_i}\right\}\\
=& \sup\left\{\lambda\geq0\mid\exists q\in[0,1]^d: \,\lambda\leq\min_{\nu\in\{0,1\}^d} P(\nu)f_\nu(q)\right\}\\
=&\sup_{q\in[0,1]^d}\,\min_{\nu\in\{0,1\}^d} P(\nu)f_\nu(q)\\
=&\max_{q\in[0,1]^d}\,\min_{\nu\in\{0,1\}^d} P(\nu)f_\nu(q),
\end{align*}
where for the middle identity we have used that $P>0$, which prevents the possibility of dealing with anomalous products of the form $0\cdot(+\infty)$, and for the last identity we have used Lemma~\ref{lem:cont}, and that $[0,1]^d$ is compact.

But observe that for each $q\in[0,1]^d$ there must exists an $\omega$ which minimizes (possibly with ties) the quantity $P(\nu)f_\nu(q)$, with $\nu\in\{0,1\}^d$; in particular, $[0,1]^d\subset\cup_{\omega\in\{0,1\}^d}\mathcal{Q}_\omega$. In particular, since $\mathcal{Q}_\omega\subset[0,1]^d$ for each $\omega$, we obtain that \[[0,1]^d=\bigcup_{\omega\in\{0,1\}^d}\mathcal{Q}_\omega.\]
Consequently, from the last identity for $\lambda(P)$, we finally obtain that
\[\lambda(P)=\max_{\omega\in\{0,1\}^d}\,\max_{q\in\mathcal{Q}_\omega}\,\min_{\nu\in\{0,1\}^d} P(\nu)f_\nu(q)=\max_{\omega\in\{0,1\}^d}\,\max_{q\in\mathcal{Q}_\omega} P(\omega)f_\omega(q),\]
where for the last identity we have used the defining property of the set $\mathcal{Q}_\omega$.
\end{proof}

Lemma~\ref{lem:formula} reduces the calculation of $\lambda(P)$ to $2^d$ optimization problems of the form:
\begin{equation}
\label{ide:optpro}
\max_{q\in\mathcal{Q}_\omega}P(\omega)f_\omega(q),\hbox{ for each }\omega\in\{0,1\}^d.
\end{equation} 
Our next result shows how to make each  of these problems more explicit.

\begin{lemma}
\label{lem:varx}
Assume $P>0$. For a given $\omega\in\{0,1\}^d$, the transformation $q\longrightarrow x$ with $x=(x_1,\ldots,x_d)$ and $x_i:=(\frac{q_i}{1-q_i})^{1-2\omega _i}$, is a bijection between $(0,1)^d$ and $\mathbb{R}_+^d$, and in terms of the variable $x$:
\begin{equation}
f_\omega(q)=\prod_{i=1}^d (1+x_i).
\label{ide:fomegaq}
\end{equation}
In particular, for $q\in(0,1)^d$:
\begin{align}
q\in\mathcal{Q}_\omega\hbox{ if and only if }\forall\nu\in\{0,1\}^d:\,\prod_{i:\,\nu_i\neq\omega _i}x_i\leq \frac{P(\nu)}{P(\omega)},
\label{eq:nonlinconst}
\end{align}
where $\prod_{i:\,\nu_i\neq\omega _i}x_i:=1$ when  $\nu=\omega$.
\end{lemma}
\begin{proof}
If $\omega_i=1$ then $x_i = \frac{1-q_i}{q_i}$, which is a strictly increasing function of $q_i$. Instead, if $\omega_i=0$ then $x_i = \frac{q_i}{1-q_i}$, which is a strictly decreasing function of $q_i$. Thus, in either case, $x_i$ is a strictly monotone function of $q_i$, with range $(0,+\infty)$ when $q_i\in(0,1)$. From this it is immediate that the transformation $q\rightarrow x$ from $(0,1)^d$ to $\mathbb{R}_+^d$ is one-to-one and onto.

On the other hand, if $\omega_i=1$ then  $q_i=\frac{1}{1+x_i}$, hence
\[q_i^{-\omega_i}(1-q_i)^{\omega_i-1}=\frac{1}{q_i}=1+x_i.\]
Likewise, if $\omega_i=0$ then $q_i=\frac{x_i}{1+x_i}$ i.e. $(1-q_i)=\frac{1}{1+x_i}$, hence
\[q_i^{-\omega_i}(1-q_i)^{\omega_i-1}=\frac{1}{1-q_i}=1+x_i.\]
In either case, $q_i^{-\omega_i}(1-q_i)^{\omega_i-1}=(1+x_i)$, which implies the identity in equation (\ref{ide:fomegaq}).

Because $P>0$, observe for $q\in(0,1)^d$ that:
\begin{align}
\label{ide:fratio}
q\in\mathcal{Q}_\omega\hbox{ if and only if }\forall\nu\in\{0,1\}^d:\,\frac{f_\omega(q)}{f_\nu(q)}\leq \frac{P(\nu)}{P(\omega)}.
\end{align}
But, in terms of the original variable $q$:
\[\frac{f_\omega(q)}{f_\nu(q)} = \prod_{i=1}^d q_i^{\nu_i-\omega_i}(1-q_i)^{\omega_i-\nu_i}. \]
Note however that if $\omega_i=\nu_i$ then $q_i^{\nu_i-\omega_i}(1-q_i)^{\omega_i-\nu_i}=1$. If instead $\omega_i\ne\nu_i$, there are only two possibilities. On the one hand, if $\omega_i=0$ and $\nu_i=1$, then
\[q_i^{\nu_i-\omega_i}(1-q_i)^{\omega_i-\nu_i} = \frac{q_i}{1-q_i}=\left(\frac{q_i}{1-q_i}\right)^{1-2\omega_i}=x_i.\]
On the other hand, if $\omega_i=1$ and $\nu_i=0$, then 
\[q_i^{\nu_i-\omega_i}(1-q_i)^{\omega_i-\nu_i} = \frac{1-q_i}{q_i}=\left(\frac{q_i}{1-q_i}\right)^{1-2\omega_i}=x_i.\]
As a result:
\[\frac{f_\omega(q)}{f_\nu(q)} = \prod_{i:\nu_i\neq\omega_i}x_i,\]
which together with equation~(\ref{ide:fratio}) implies the lemma.
\end{proof}

Using the variable $x$ in $\mathbb{R}_+^d$ instead of $q$ in $(0,1)^d$ has two advantages in terms of the optimization problems in equation~(\ref{ide:optpro}). First, up to the factor $P(\omega)$, the objective function does not depend on $\omega$. Second, the objective function is monotonically increasing in each coordinate of $x$; in particular, any maximum must lie on the boundary of the feasible region, i.e. at least one of the inequalities in equation~(\ref{eq:nonlinconst})  must be an equality. Nevertheless, the special nature of the constraints, suggests introducing the additional change of variables $x\to y$, with $y=(y_1,\ldots,y_d)$ and $y_i:=\ln(x_i)$, which is clearly a bijection between $\mathbb{R}_+^d$ and $\mathbb{R}^d$. The following result is now immediate from the previous lemma.


\begin{corollary}
\label{cor:vary}
For a given $\omega\in\{0,1\}^d$, the transformation $q\longrightarrow y$ with $y=(y_1,\ldots,y_d)$ and $y_i:=(1-2\omega_i)\cdot\ln\left(\frac{q_i}{1-q_i}\right)$, is a bijection between $(0,1)^d$ and $\mathbb{R}^d$, and in terms of the variable $y$:
\[f_\omega(q)=\prod_{i=1}^d (1+e^{y_i}).\]
In particular, for $q\in(0,1)^d$:
\begin{align}
q\in\mathcal{Q}_\omega\hbox{ if and only if }\forall\nu\in\{0,1\}^d:\,\sum_{i:\,\nu_i\neq\omega _i}y_i\leq \ln\left(\frac{P(\nu)}{P(\omega)}\right),
\label{eq:linconst}
\end{align}
where $\sum_{i:\,\nu_i\neq\omega _i}x_i:=0$ when $\nu=\omega$.
\end{corollary}

Using the variable $y$ instead of $q$ retains all the good properties we already had with the variable $x$, particularly, the objective function remains monotonically increasing in each coordinate, however, it also transforms the feasible region into a polyhedron~\cite[Chapter 8]{polyhedra}, which is a well-studied geometric object. We show how to exploit this geometry in the next section.

\subsection{Geometric Insights}

In this section, we fix an outcome $\omega\in\{0,1\}^d$ and describe a combinatorial algorithm to solve the associated optimization problem in equation~(\ref{ide:optpro}).

In what follows, all vectors are represented as column vectors. As seen in equation~(\ref{eq:linconst}), in terms of the variable $y$, the feasible set $\mathcal Q_\omega$ can be formulated as a linear inequality system in standard form. The following result is now immediate from Corollary~\ref{cor:vary}.

\begin{corollary}
\label{cor:linIneqSys}
Assume that $P>0$. For a given $\omega\in\{0,1\}^d$, let $A_\omega$ be the binary matrix of dimensions $(2^d-1)\times d$ with entries $A_\omega(\nu,i):=\llbracket\nu_i\neq \omega_i\rrbracket$, for each $\nu\in\{0,1\}^d\setminus\{\omega\}$ and $i\in\{1,\ldots,d\}$. Furthermore, let $b_\omega$ be a column vector of dimension $(2^d-1)$ with entries $b_\omega(\nu):=\log(P(\nu)/P(\omega))$, for each $\nu\in\{0,1\}^d\setminus\{\omega\}$. Then, in terms of the variable $y$, $\mathcal Q_\omega$ corresponds to the set of $y\in\mathbb R^d$ satisfying the coordinatewise inequalities:
\begin{align}
\label{eq:linIneqSys}
A_\omega y\leq b_\omega.
\end{align}
\end{corollary}

The above inequality characterizes $\mathcal{Q}_\omega$ (in terms of the variable $y$) as a non-empty convex polyhedron in $\mathbb R^d$. Recall, $y\in\mathcal{Q}_\omega$ is called a \emph{vertex} if there exists an invertible sub-matrix $A_\omega'$ of $A_\omega$ of dimensions $d\times d$ and a corresponding sub-vector $b_\omega'$ of $b_\omega$ of dimension $d$ such that $A_\omega'y=b_\omega'$~\cite[Chapter 8, equation (23)]{polyhedra}. (The sub-matrix $A_\omega'$ and the sub-vector $b_\omega'$ are associated with the same rows of $A_\omega$ and $b_\omega$, respectively.)

\begin{lemma}
The polyhedron in equation~(\ref{eq:linIneqSys}) is pointed, i.e. it contains at least one vertex.
\end{lemma}
\begin{proof}
For each $i\in\{1,\ldots,d\}$, let $\nu_i\in\{0,1\}^d$ be such that $\nu_i(j)=\omega(j)$ for $j\ne i$, and $\nu_i(i)=1-\omega(i)$. Then the sub-matrix of $A_\omega$ associated with rows in the set $\{\nu_1,\ldots,\nu_d\}$ corresponds to the $(d\times d)$ identity matrix. As a result, the kernel of $A_\omega$---which coincides exactly with the so-called ``lineality space'' of the polyhedron---is $\{0\}$, which implies that the polyhedron is pointed~\cite[Chapter 8, equations (6) and (23)]{polyhedra}.
\end{proof}

In the language of polyhedral programming, a vertex is a zero-dimensional face. More generally, if $c\in\mathbb R^d\setminus\{0\}$, $\delta\in\mathbb R$, and $G:=\{y\in\mathbb R^d \mid c^t y=\delta\}$ we say the affine hyperplane $G$ is a \emph{supporting hyperplane} of $\mathcal Q_\omega$ at the point $y\in\mathcal Q_\omega$ if $y\in G\cap \mathcal Q_\omega$ and $\mathcal Q_\omega$ is contained in one of the closed half-spaces bounded by $G$~\cite[p.~20]{multiobjectiveLP}. The non-empty set $F:=G\cap\mathcal Q_\omega$ is called a \emph{face} of $\mathcal Q_\omega$.  
Equivalently, a face of $\mathcal{Q}_\omega$ is any set of the form $\{y\in\mathcal Q_\omega\mid A'_\omega y=b'_\omega\}$, where $A'_\omega$ and $b'_\omega$ are a sub-matrix and sub-vector associated with the same rows of $A_\omega$ and $b_\omega$, respectively~\cite[Theorem 2.3.3]{multiobjectiveLP}. (Here, $A_\omega'$ does not need to be a square matrix.) The dimension of a face $F$ associated with the subsystem $A'_\omega y = b'_\omega$ is $d-\text{rank}(A'_\omega)$.

\begin{corollary}
\label{cor:verts}
If $y\in\partial\mathcal Q_\omega$, the boundary of $Q_\omega$, and $y$ is not a vertex of $\mathcal Q_\omega$, then $y$ lies in the relative interior of \emph{some} positive-dimensional face of $Q_\omega$. That is, there is a positive-dimensional face $F$ and some $\epsilon>0$ such that the intersection of the closed $\epsilon$-ball around $y$ and the affine hull of $F$ is contained in $F$.
\end{corollary}
\begin{proof}
First, $\mathcal Q_\omega$ equals the union of the relative interiors of its faces, which are disjoint~\cite[Corollary 2.3.7]{multiobjectiveLP}. In particular:
\begin{align*}
\partial Q_\omega&=\underset{\text{faces }F\subsetneqq\mathcal Q_\omega}{\sqcup} \text{relint}(F)\\
&=\left(\underset{\text{non-vertex faces }F\subsetneqq\mathcal Q_\omega}{\sqcup} \text{relint}(F)\right) \sqcup\left(\underset{\text{vertices }v\in\mathcal Q_\omega}{\sqcup}\{v\}\right),
\end{align*}
where $\text{relint}(\cdot)$ denotes the relative interior, and $\sqcup$ denotes a disjoint union. In particular, since a face coincides with its own relative interior if and only if it is a vertex, if $y\in\partial\mathcal Q_\omega$ but $y$ is not a vertex then $y$ must lie in the relative interior of a unique positive-dimensional face.
\end{proof}

Next we address the optimization problem in equation~(\ref{ide:optpro}) for a fixed $\omega\in\{0,1\}^d$. 
Hereafter, we abuse notation slightly and define \[f_\omega(y):=\prod_{i=1}^d(1+e^{y_i}),\]
to denote the reparameterized version of the objective function $f_\omega(q)$. The following result rules out points in the relative interior of positive-dimensional faces as maximizers of $f_\omega$.



\begin{lemma}
\label{lem:critPoints}
Let $F\subset\mathcal Q_\omega$ denote a positive-dimensional face of the polyhedron, and $\hat y$ denote a point in the relative interior of $F$. Then $f_\omega(\hat y)<\max_{y\in\mathcal Q_\omega}f_\omega(y)$. More specifically:
\begin{enumerate}
\item If the gradient $\nabla f_\omega(\hat y)$ is not orthogonal to $F$, then $f_\omega$ can be strictly increased on $F$, that is, there is some $\hat z\in F$ such that $f_\omega(\hat z)>f_\omega(\hat y)$. 
\item If the gradient $\nabla f_\omega(\hat y)$ is orthogonal to $F$, then $f_\omega(\hat y)$ is a local minimum on $F$.
\end{enumerate}
\end{lemma}
\begin{proof}
Clearly, $f_\omega$ is analytic, in particular, it has continuous partial derivatives of any order. 

First observe that
\[\frac{\partial  f_\omega}{\partial y_i}(y) = e^{y_i}\prod_{j\neq i}(1+e^{y_j}) = f_\omega(y)\,\frac{e^{y_i}}{1+e^{y_i}}.\]
Therefore, if $y\rightarrow\gamma_i:=\frac{e^{y_i}}{1+e^{y_i}}$, and $y\rightarrow\gamma$ is the  transformation defined as $\gamma=\left(\gamma_1,\ldots,\gamma_d\right)^t$ then
\[\nabla f_\omega(y)=f_\omega(y) \,\gamma,\]
which implies that $\nabla f_\omega(y)\ne0$, for all $y\in\mathbb{R}^d$. In particular, if $\nabla f_\omega(y)$ is not orthogonal to $F$, a small perturbation in the direction of the  projection of $\nabla f_\omega(y)$ onto $F$ will increase $f_\omega$. This shows the first statement in the lemma.

On the other hand: 
\[\frac{\partial^2 f_\omega}{\partial y_i^2}(y) = e^{y_i}\prod_{j\neq i}(1+e^{y_j})=f_\omega(y)\,\gamma_i,\]
and for $i\ne j$:
\[\frac{\partial^2f_\omega}{\partial y_j y_i}(y) = e^{y_i}e^{y_j}\prod_{k\neq i,j}(1+e^{y_k}) = f_\omega(y)\,\gamma_i\gamma_j.\]
As a result, $\nabla^2 f_\omega(y)$, the Hessian matrix of $f_\omega$ at $y$, admits the decomposition:
\[\nabla^2 f_\omega(y) = f_\omega(y)\,\big(\Gamma_1+\Gamma_2\big),\]
where $\Gamma_1 := 
\hbox{diag}\big(\gamma_1(1-\gamma_1),\dots,\gamma_d(1-\gamma_d)\big)$, and
\[\Gamma_2 := 
\begin{bmatrix}
\gamma_1^2 & \gamma_1\gamma_2 &  \dots  & \gamma_1\gamma_d\\
\gamma_2\gamma_1 & \gamma_2^2 & \dots & \gamma_2\gamma_d \\
\vdots & \vdots & \ddots & \vdots\\
\gamma_d\gamma_1 & \gamma_d\gamma_2 & \dots & \gamma_d^2
\end{bmatrix}=\gamma\gamma^t.
\]
Because each $0<\gamma_i<1$, $\Gamma_1$ is strictly positive semidefinite. Since $\Gamma_2$ is positive definite, and $f_\omega(y)>0$ for all $y\in\mathbb{R}^d$, $\nabla^2 f_\omega(y)$ is strictly positive definite regardless of $y$. As a result, if $\nabla f_\omega(\hat y)$ is orthogonal to $F$, then $\hat y$ is a local minimum of $f_\omega$ along $F$. This completes the  proof of the lemma.
\end{proof}

Finally, combining  Corollary~\ref{cor:verts} and Lemma~\ref{lem:critPoints}, we obtain the following central result, which implies that the maxima in~(\ref{ide:optpro}) can occur only occur among a finite number of well-characterized points in $Q_\omega$.

\begin{theorem}
If $P>0$ then, for each $\omega\in\{0,1\}^d$, the maximum $\max\limits_{y\in\mathcal Q_\omega} f_\omega(y)$, can only occur at a vertex of $\mathcal Q_\omega$.
\end{theorem}

\section{Algorithms for $\lambda(P)$}

Computing $\lambda(P)$ requires solving the optimization problem~(\ref{ide:optpro}) for each of the $2^d$ outcomes. As previously described, solving each optimization problem can be achieved by evaluating $f_\omega$ at each vertex of $\mathcal Q_\omega$, and the vertices can be found as unique solutions of invertible $(d\times d)$-subsystems $A'_\omega y=b'_\omega$. This motivates Algorithm~\ref{alg1}, which computes $\lambda(P)$ by exploring square subsystems of $A_\omega y\leq b_\omega$ to find vertices, evaluating $f_\omega(y^*)$ at each vertex $y^*$ for each outcome $\omega$, and returning the largest of these.

\begin{algorithm} 
\caption{A na\"ive exact algorithm for $\lambda(P)$} 
\label{alg1} 
\begin{algorithmic} 
    \State {$M \Leftarrow 0$}
    \For{$\omega\in\{0,1\}^d$}
        \State {$A_\omega \Leftarrow \{\llbracket \omega_i\neq \nu_i\rrbracket_{i=1,\ldots,d} \}_{\nu\in\{0,1\}^d\setminus\{\omega\}}$}
        \State {$b_\omega\Leftarrow \{\log(\frac{P(\nu)}{P(\omega)})\}_{\nu\in\{0,1\}^d\setminus\{\omega\}}$}
        \For{$\{\nu^{(1)},\nu^{(2)},\ldots,\nu^{(d)}\}\subset\{0,1\}^d\setminus\{\omega\}$}
        
            \State {$A'_\omega\Leftarrow \{\llbracket \omega_i\neq \nu_i^{(j)}\rrbracket_{i=1,\ldots,d}\}_{j=1,\ldots,d}$}
            \State {$b'_\omega\Leftarrow \{\log(\frac{P(\nu^{(j)})}{P(\omega)})\}_{j=1,\ldots,d}$}
            \If{$A'_\omega$ is invertible}
                \State {$y*\Leftarrow (A'_\omega)^{-1} b'_\omega$}
                \If{$A_\omega y'^*\leq b_\omega$}
                    \State {$M\Leftarrow M\vee f_\omega(y'^*)$}
                \EndIf
            \EndIf
        \EndFor
    \EndFor
    \\
\Return{$\lambda(P)\Leftarrow M$}
\end{algorithmic}
\end{algorithm}

For each outcome $\omega$, there are $\binom{2^d-1}{d}$ subsystems $A'_\omega y=b'_\omega$ of size $(d\times d)$ to check. For each subsystem $A'_\omega y=b'_\omega$, simple Gaussian elimination will find a unique solution, if it exists, in $O(d^3)$ time, and often terminates in less time if $A'_\omega$ is singular. If $y'$ is a unique solution to the square subsystem $A'_\omega y'=b'_\omega$, it takes $O(d2^d)$ operations to check that $y'$ is feasible, i.e., $A_\omega y'\leq b_\omega$. If $y'$ is infeasible, it often takes many fewer operations to confirm this.

Taking these operations together, and using the well-known bound on binomial coefficients, $\binom{n}{k}<(\frac{n\cdot e}{k})^k$, in the worst case there are 
$O\left(d^4(\frac{2^{d+1}e}{d})^d\right)$ operations required to compute $\lambda(P)$. The memory required by this algorithm grows much less slowly, as $O(2^d)$, if square subsystems are iterated without loading every set of $d$ indices into memory. This is common in standard combinatorial software like the \texttt{itertools} module in Python~\cite[Section 3.2]{python3standard}. In practice, we find that without any parallelization strategies and without supercomputing resources, it is feasible to compute $\lambda(P)$ for binary sources up to dimension $d=6$ by na\"ively searching for vertices.

We note that specialized algorithms to explore only those subsystems $A'_\omega y = b'_\omega$ which are invertible, and ignore singular subsystems, are still unlikely to allow computation of $\lambda(P)$ in very high dimensions. In fact, the number of invertible submatrices $A'_\omega$ of dimension $d$ has previously been recognized as a noteworthy sequence~\cite{oeis}. This sequence is hard to compute explicitly, but appears to grow exponentially fast. In fact, there are approximately $2.52\times10^{14}$ invertible subsystems in only $8$ binary dimensions~\cite{small01matrices}.

Specialized polyhedral programming algorithms may help to accelerate computation of $\lambda(P)$. For example, the vertex enumeration algorithm given in~\cite{vertexenumeration}, runs in $O(d2^dV)$ time, where $V$ is the number of vertices of $\mathcal Q_\omega$. The number of vertices is hard to characterize (it depends on $b_\omega$), but based on simulation we believe it is typically much smaller than the number of invertible subsystems. We believe a pivoting method similar to~\cite{vertexenumeration} can be adapted to take advantage of $A_\omega$'s binary structure. 

Some readers may note that each optimization program \[\max f_\omega(y)\text{ s.t. }A_\omega y\leq b_\omega\]
resembles a linear program. However, our objective function $f_\omega$ is nonlinear, and therefore linear programming techniques such as Dantzig's simplex algorithm~\cite[Chapter 5]{simplexmethod} are not suitable. Moreover, positive definiteness of the Hessian derived in Lemma~\ref{lem:critPoints} implies that $f_\omega$ is strictly convex. Although the feasible set is also convex, the fact that we seek to \emph{maximize} $f_\omega$ means most nonlinear convex programming techniques cannot be guaranteed to converge to true maxima. 

Due to the aforementioned difficulties in the combinatorial approach in high dimensions, we have also explored numerical approximation of each optimization program using nonlinear algorithms including sequential gradient-free linear approximation (COBYLA)~\cite{COBYLA}, and sequential quadratic programming (SLSQP)~\cite{SLSQP}. These show some promise but tend to suffer from numerical instability in moderate to high dimensions (above $d=5$ or so). However, because the structure of $f_\omega$ makes computing higher-order derivatives very straightforward, it may be possible to devise a specialized interior point method that makes approximating $\lambda(P)$ efficient even in higher dimensions.

\section{Proof of Concept}

Consider the Markov network~\cite{probGraphMod} in Figure~\ref{fig:MRF}, borrowed from~\cite[Chapter 2]{graphmod}. In this setting, undirected edges represent interactions in a social network of four patients, each of whom may or may not have tuberculosis (represented as four Bernoulli random variables $T_1,\ldots,T_4$). Here, the complete subgraphs (cliques) of the Markov network are $\{T_1\},\{T_2\},\{T_3\},\{T_4\},\{T_1,T_2\},\{T_1,T_3\},\{T_2,T_4\}$, and $\{T_3,T_4\}$. To each clique $C$ we associate a \emph{factor} $w:\{0,1\}^{|C|}\to\mathbb R_+$, and we associate to each configuration $(t_1,t_2,t_3,t_4)\in\{0,1\}^4$ of sick and healthy patients, the probability: 
\[P(t_1,t_2,t_3,t_4) \propto w(t_1)\cdot w(t_2)\cdot w(t_3)\cdot w(t_4)\cdot w(t_1,t_2)\cdot w(t_1,t_3)\cdot w(t_2,t_4)\cdot w(t_3,t_4).\]

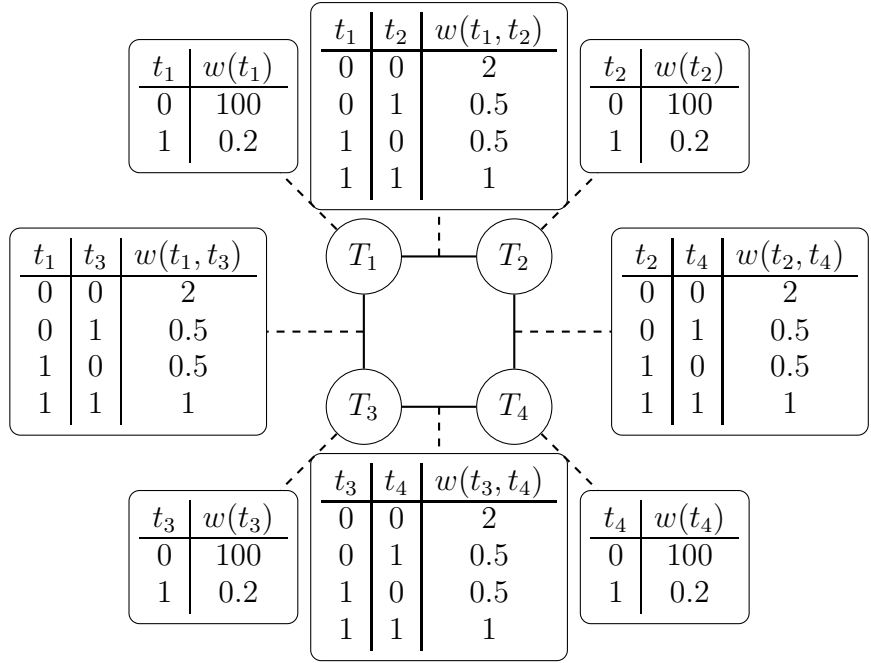
\begin{figure}
\centering
\begin{tikzpicture}[node distance=2cm]
\node (T1)[RV]{$T_1$};
\node (T2)[RV,right of=T1]{$T_2$};
\node (T3)[RV,below of=T1]{$T_3$};
\node (T4)[RV,below of=T2]{$T_4$};
\node (T1d)[dist,above of=T1,xshift=-2cm]{\begin{tabular}{c|c}
    $t_1$ & $w(t_1)$ \\
    \hline
    $0$ & $100$\\
    $1$ & $0.2$
\end{tabular}};
\node (T2d)[dist,above of=T2,xshift=2cm]{\begin{tabular}{c|c}
    $t_2$ & $w(t_2)$ \\
    \hline
    $0$ & $100$\\
    $1$ & $0.2$
\end{tabular}};
\node (T3d)[dist,below of=T3,xshift=-2cm]{\begin{tabular}{c|c}
    $t_3$ & $w(t_3)$ \\
    \hline
    $0$ & $100$\\
    $1$ & $0.2$
\end{tabular}};
\node (T4d)[dist,below of=T4,xshift=2cm]{\begin{tabular}{c|c}
    $t_4$ & $w(t_4)$ \\
    \hline
    $0$ & $100$\\
    $1$ & $0.2$
\end{tabular}};
\coordinate[right of=T1,xshift=-1cm](T12);
\coordinate[right of=T3,xshift=-1cm](T34);
\coordinate[below of=T1,yshift=1cm](T13);
\coordinate[below of=T2,yshift=1cm](T24);
\node (T12d)[dist,above of=T12]{\begin{tabular}{c|c|c}
    $t_1$ & $t_2$ & $w(t_1,t_2)$ \\
    \hline
    $0$ & $0$ & $2$\\
    $0$ & $1$ & $0.5$\\
    $1$ & $0$ & $0.5$\\
    $1$ & $1$ & $1$
\end{tabular}};
\node (T34d)[dist,below of=T34]{\begin{tabular}{c|c|c}
    $t_3$ & $t_4$ & $w(t_3,t_4)$ \\
    \hline
    $0$ & $0$ & $2$\\
    $0$ & $1$ & $0.5$\\
    $1$ & $0$ & $0.5$\\
    $1$ & $1$ & $1$
\end{tabular}};
\node (T13d)[dist,left of=T13,xshift=-1cm]{\begin{tabular}{c|c|c}
    $t_1$ & $t_3$ & $w(t_1,t_3)$ \\
    \hline
    $0$ & $0$ & $2$\\
    $0$ & $1$ & $0.5$\\
    $1$ & $0$ & $0.5$\\
    $1$ & $1$ & $1$
\end{tabular}};
\node (T24d)[dist,right of=T24,xshift=1cm]{\begin{tabular}{c|c|c}
    $t_2$ & $t_4$ & $w(t_2,t_4)$ \\
    \hline
    $0$ & $0$ & $2$\\
    $0$ & $1$ & $0.5$\\
    $1$ & $0$ & $0.5$\\
    $1$ & $1$ & $1$
\end{tabular}};
\draw [line] (T1)--(T2);
\draw [line] (T1)--(T3);
\draw [line] (T2)--(T4);
\draw [line] (T3)--(T4);
\draw [desc] (T1)--(T1d);
\draw [desc] (T2)--(T2d);
\draw [desc] (T3)--(T3d);
\draw [desc] (T4)--(T4d);
\draw [desc] (T12)--(T12d);
\draw [desc] (T13)--(T13d);
\draw [desc] (T24)--(T24d);
\draw [desc] (T34)--(T34d);
\end{tikzpicture}
\caption{Markov network that models the interaction of four hypothetical patients that may or may not have tuberculosis. Patient $i$ is healthy if $T_i=0$, and infected if $T_i=1$.}
\label{fig:MRF}
\end{figure}

This network reflects the intuition that, if one patient who has tuberculosis interacts with another, it is more likely for the latter to have tuberculosis. In fact, the joint distribution $P$ of $(T_1,T_2,T_3,T_4)$ is exchangeable (labels on the patients can be permuted without affecting the joint probability of their tuberculosis status). Using Algorithm~\ref{alg1}, we find that $\lambda(P)$ is very close to $1$. We transform a vertex $y^*$ which achieves $\lambda(P)$ back to a probability $q^*$ and find explicitly: 
\[P = 0.999999\cdot Be(0.000125)\otimes Be(0.000125)\otimes Be(0.000125)\otimes Be(0.000125)+0.0000001\cdot R,\]
where $R$ is a residual probability distribution with low entropy ($\approx 2$ bits, compared to the uniform distribution over $\{0,1\}^4$, which has $4$ bits of entropy). This means that, despite the dependence implied by the interactions, a large fraction of the time it will appear as though the patients are infected with tuberculosis independently, each with a very small probability of infection. 

It is not always the case that a source represented by a probabilistic graphical model has a large independent weight. Consider a simpler version of the Markov network, shown in Figure~\ref{fig:MRF2x2}. In this case, a non-negligible fraction of the data produced by the source cannot be recapitulated by an independent model. Let $P$ denote the joint distribution of $(T_1,T_2)$. Using Algorithm~\ref{alg1}, we find that $\lambda(P)=0.817$. Moreover,
\[P = 0.817\cdot Be(0.048)\otimes Be(0.048)+0.183\cdot \delta_{(1,1)}.\]
That is, a large fraction of the time a realization of these two patients' tuberculosis states cannot be attributed to the largest independent component of $P$. 

These two examples demonstrate how scientists and engineers may benefit from detecting a source's independent weight. If a source under study is known to have $\lambda(P)\approx1$, even if the source fails a hypothesis test of independence, the modeler might save considerable complexity while still recapitulating most of the features of the source. In contrast, if a source has very low independent weight, the scientist could find meaningful mechanistic insights in the residual component, such as in the latter example, where a sample originates \emph{either} from a hidden independent model or a deterministic one.

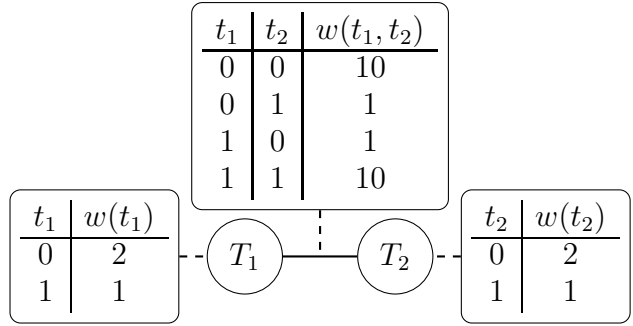
\begin{figure}
\centering
\begin{tikzpicture}[node distance=2cm]
\node (T1)[RV]{$T_1$};
\node (T2)[RV,right of=T1]{$T_2$};
\coordinate[right of=T1,xshift=-1cm](T12);
\node (T1d)[dist,left of=T1]{\begin{tabular}{c|c}
    $t_1$ & $w(t_1)$ \\
    \hline
    $0$ & $2$\\
    $1$ & $1$
\end{tabular}};
\node (T2d)[dist,right of=T2]{\begin{tabular}{c|c}
    $t_2$ & $w(t_2)$ \\
    \hline
    $0$ & $2$\\
    $1$ & $1$
\end{tabular}};
\node (T12d)[dist,above of=T12]{\begin{tabular}{c|c|c}
    $t_1$ & $t_2$ & $w(t_1,t_2)$ \\
    \hline
    $0$ & $0$ & $10$\\
    $0$ & $1$ & $1$\\
    $1$ & $0$ & $1$\\
    $1$ & $1$ & $10$
\end{tabular}};
\draw[line](T1)--(T2);
\draw[desc](T1d)--(T1);
\draw[desc](T2d)--(T2);
\draw[desc](T12d)--(T12);
\end{tikzpicture}
\caption{Markov network that models the interaction of two hypothetical patients which may or may not have tuberculosis. In this setting, the marginal probability of a patient being infected with tuberculosis is moderate ($\approx22\%$), and the probability that exactly one of the two patients is infected is relatively low ($\approx8\%$). This might be a realistic model for, e.g., two inmates sharing a cell in a prison with a tuberculosis outbreak.}
\label{fig:MRF2x2}
\end{figure}



\bibliographystyle{abbrv}

\bibliography{biblio}

\begin{thebibliography}{10}

\bibitem{vertexenumeration}
D.~Avis and K.~Fukuda.
\newblock A pivoting algorithm for convex hulls and vertex enumeration of
  arrangements and polyhedra.
\newblock {\em Discrete {\&} Computational Geometry}, 8(3):295--313, Sep 1992.

\bibitem{simplexmethod}
G.~B. Dantzig.
\newblock {\em {Linear Programming and Extensions}}.
\newblock United States Air Force Project RAND. The RAND Corporation, Aug.
  1963.

\bibitem{oeis}
Y.~Dekel.
\newblock Number of real regular n x n (0,1) matrices modulo rows permutations,
  2003.
\newblock In: The On-line Encyclopedia of Integer Sequences.

\bibitem{python3standard}
D.~Hellmann.
\newblock {\em The Python 3 Standard Library by Example}.
\newblock Addison-Wesley Professional, 1 edition, 2017.

\bibitem{probGraphMod}
D.~Koller and N.~Friedman.
\newblock {\em Probabilistic Graphical Models: Principles and Techniques -
  Adaptive Computation and Machine Learning}.
\newblock The MIT Press, 2009.

\bibitem{SLSQP}
D.~Kraft.
\newblock {\em A Software Package for Sequential Quadratic Programming}.
\newblock Deutsche Forschungs- und Versuchsanstalt f{\"u}r Luft- und Raumfahrt
  K{\"o}ln: Forschungsbericht. Wiss. Berichtswesen d. DFVLR, 1988.

\bibitem{multiobjectiveLP}
D.~Luc.
\newblock {\em Multiobjective linear programming: An Introduction}.
\newblock Springer, 01 2015.

\bibitem{PeaLla20}
A.~Pearson and M.~E. Lladser.
\newblock {On Contamination of Symbolic Datasets}.
\newblock (Submitted).

\bibitem{COBYLA}
M.~Powell.
\newblock A view of algorithms for optimization without derivatives.
\newblock {\em Mathematics TODAY}, 43, 01 2007.

\bibitem{polyhedra}
A.~Schrijver.
\newblock {\em Theory of Linear and Integer Programming.}
\newblock Number Vol. 75 in Wiley-Interscience Series in Discrete Mathematics
  and Optimization. Wiley, 1998.

\bibitem{graphmod}
B.~Taskar and L.~Getoor.
\newblock {\em Introduction to Statistical Relational Learning.}
\newblock Adaptive Computation and Machine Learning. The MIT Press, 2007.

\bibitem{small01matrices}
M.~Zivkovi\'c.
\newblock Classification of small (0,1) matrices.
\newblock {\em Linear Algebra and its Applications}, 414(1):310 -- 346, 2006.

\end{thebibliography}

\end{document}